\documentclass[11pt,a4paper]{article}
\usepackage[english]{babel}
\usepackage[utf8]{inputenc}
\usepackage[T1]{fontenc}
\usepackage{amsmath}
\usepackage{amsthm}
\usepackage{amssymb, color}
\usepackage[a4paper,top=1.5in,bottom=1.5in,left=1in,right=1in,marginparwidth=1.75cm]{geometry}
\usepackage{graphicx}
\usepackage{enumerate}
\usepackage{dsfont}
\usepackage{mathrsfs}
\usepackage{varioref}
\usepackage{hyperref}
\usepackage[capitalise]{cleveref}
\usepackage{hyperref}
\usepackage{todonotes}
\usepackage{array}
\usepackage{makecell}
\usepackage{mathtools}
\usepackage{comment}

\theoremstyle{plain}
\newtheorem{theorem}{Theorem}[section]
\newtheorem{proposition}[theorem]{Proposition}

\theoremstyle{definition}
\newtheorem{definition}[theorem]{Definition}
\newtheorem{example}[theorem]{Example}
\newtheorem{remark}[theorem]{Remark}

\newcommand{\half}{\frac{1}{2}}
\newcommand{\R}{\mathbb{R}}

\newcommand{\N}{\mathbb{N}}

\newcommand{\bV}{\textbf{V}}

\newcommand{\cM}{\mathcal{M}}

\newcommand{\cC}{\mathcal{C}}

\newcommand{\sB}{\mathscr{B}}

\newcommand{\da}{\delta}
\newcommand{\s}{\sigma}
\newcommand{\e}{\varepsilon}

\numberwithin{equation}{section}

\newcommand{\de}{\textnormal{ d}}
\newcommand{\den}{\textnormal{d}}

\newcommand{\I}{\mathds{1}}

\newcommand{\squarebrac}[1]{\left[ #1 \right]}
\newcommand{\curvebrac}[1]{\left( #1 \right)}
\newcommand{\curlybrac}[1]{\left\{ #1 \right\}}
\newcommand{\modd}[1]{\left| #1 \right|}
\newcommand{\norm}[1]{\left\lVert#1\right\rVert}

\newcommand{\BV}{\textnormal{BV}}

\DeclareMathOperator*{\weak}{w-lim}
\DeclareMathOperator*{\vague}{v-lim}

\definecolor{vargreen}{RGB}{0,150,0}

\begin{document}

\title{ Vague and weak convergence of signed measures\thanks{The authors thank Lutz Mattner for helpful comments on weak convergence for general Hausdorff spaces.}}

    \author{Martin Herdegen, Gechun Liang, Osian Shelley\thanks{All authors: University of Warwick, Department of Statistics, Coventry, CV4 7AL, UK; \{m.herdegen, g.liang, o.d.shelley\}@warwick.ac.uk;}}

    \date{\today}

    \maketitle

\begin{abstract}
    Necessary and sufficient conditions for weak and vague convergence of measures are important
    for a diverse host of applications. This paper aims to give a comprehensive description of the relationship between the
    two modes of convergence when the measures are signed, which is largely absent from the literature. Furthermore, when the underlying space is $\R$, we study the relationship between vague convergence of signed measures and the pointwise convergence of their distribution functions.
\end{abstract}



\section{Introduction}

In this paper, we aim to provide necessary and sufficient conditions
for weak and vague convergence of signed measures. They lie at the
heart of key results in probability theory such as Karamata's
Tauberian theorem (see e.g.~Feller \cite[XIII.5, Theorem
1]{feller_introduction_1971}), whose proof relies on the equivalence
between the vague convergence of finite positive measures and the
pointwise convergence of their distribution functions (at continuity
points of the limiting measure). Motivated by an application in
stochastic control, we extended Karamata's theorem to signed
measures in Herdegen et al.~\cite{herdegen_tauberian_2022}. This
requires to study the relationship between vague convergence of
signed measures and pointwise convergence of their distribution
functions.

For positive measures, the relationship between weak convergence,
vague convergence and convergence of their distribution functions is
well understood; see e.g.~Dieudonné and Macdonald~\cite{dieudonne_treatise_1970}, Vere-Jones \cite{daley_introduction_2003} or Kallenberg
\cite{kallenberg_random_2017,kallenberg_foundations_2021}. However, the
conditions needed to extend this theory to the case of signed measures are seemingly absent from the literature. We
fill this gap by providing a comprehensive description of the
relationship between weak and vague convergence of signed measures
(including their Hahn-Jordan decompositions), as well as pointwise
convergence of their distribution functions.

It turns out that one of the key conditions for the equivalence of
different modes of convergence on a metrisable space is to check whether
or not mass is preserved in the limit. For example, if \emph{mass is
not lost at infinity}, then vague conference is equivalent to weak
convergence. Such a mass preserving condition is usually referred as
tightness condition in the literature. We further show that if
\emph{mass is not lost on compact sets}, then vague convergence
implies the convergence of the positive and negative parts in the
Hahn-Jordan decomposition. Moreover, if \emph{mass is not lost
globally}, then weak convergence of the positive and negative parts
in the Hahn-Jordan decomposition also holds. These results are
summarised in \Cref{table:1}.

When restricted to $\R$, we also provide necessary and
sufficient conditions for the equivalence between vague convergence
of signed measures and pointwise convergence of their distribution
functions (at continuity points of the limiting measure). To this end, we propose a new type of \emph{local (zero) mass preserving
condition} in \Cref{def:no_mass}. It prevents the positive and negative parts of the
singular decompositions to cancel in the limit. Using this new condition, we
give in Theorem \ref{thm:vague_F_equivalence} a clear characterisation of the relationship between vague convergence of signed measures and pointwise convergence of their distribution functions.

\subsection{The definition of vague and weak convergence}\label{subsection:notation_and_defnitinions}

Throughout the paper, let $\Omega$ be a metrisable space and $\sB({\Omega})$ its Borel $\sigma$-algebra.

Let $C(\Omega)$ be the space of all continuous $\R$-valued functions on $\Omega$, $C_b(\Omega)$ the subspace of all $f \in C(\Omega)$ such that $f$ is bounded, $C_0(\Omega)$ the subspace of all $f \in C(\Omega)$ such that for any $\e>0$, there exists a compact set $K_\e\in \sB(\Omega)$ with $\modd{f} < \e$ on $K_\e^c$, and $C_c(\Omega)$ the subspace of all $f \in C(\Omega)$ such that  $f$ has compact support. We have the inclusions $C_c(\Omega) \subseteq C_0(\Omega)\subseteq C_b(\Omega) \subseteq C(\Omega)$.

For a signed measure $\mu$ on $(\Omega, \sB({\Omega}))$, we denote its Hahn-Jordan decomposition by $\mu =\mu^+ - \mu^-$, and its associated variation measure by $\modd{\mu}:= \mu^+ + \mu^-$. The \emph{total variation} of a signed measure $\mu$ is denoted by $\norm{\mu}:= \modd{\mu}(\Omega)$, and we say that $\mu$ is \emph{finite} if $\norm{\mu} < \infty$.

A finite signed measure $\mu$ on $(\Omega, \sB({\Omega}))$ is called a \emph{finite signed Radon measure} if $\modd{\mu}$ is \emph{inner regular}, i.e., for each $A \in \sB(\Omega)$,
  \begin{align*}
 \modd{\mu}(A) &= \sup\{\modd{\mu}(K): K\in \sB(\Omega), K \textnormal{ compact}, K \subset A\}.
    \end{align*}
We denote the set of all finite signed Radon measures on $(\Omega, \sB({\Omega}))$  by $\cM(\Omega)$ and the subset of all finite positive Radon measures by $\cM^+(\Omega)$.

We now come to the key definition of this paper.

\begin{definition}\label{def:vague}
For $\mu \in \cM(\Omega)$, define the map $I_\mu:C_b(\Omega) \to \R$ by
\begin{equation*}
    I_\mu(f) = \int_\Omega f \de \mu.
\end{equation*}
We say that a sequence $\{\mu_n\} \subset \cM(\Omega)$ converges to $\mu \in \cM(\Omega)$
\begin{enumerate}[\normalfont(a)]
    \item \emph{weakly} if $I_{\mu_n}(f) \to I_{\mu}(f)$ for all $f \in C_b(\Omega)$,\footnote{Weak convergence is sometimes referred to as narrow convergence; see \cite[Section 8.1]{bogachev_measure_2007}.} and we write $$\weak_{n \to \infty}\mu_n = \mu;$$
    \item \emph{vaguely}  if $I_{\mu_n}(f) \to I_{\mu}(f)$ for all $f \in C_c(\Omega)$, and we write $$\vague_{n \to \infty}\mu_n = \mu.$$
\end{enumerate}
\end{definition}

Before making some comments on our definition of vague convergence, it is useful to recall the famous Riesz-Markov-Kakutani Representation Theorem; see \cite[Theorem 14.14]{aliprantis_infinite_1999} for a proof.

\begin{theorem}[Riesz-Markov-Kakutani Representation Theorem]\label{thm:Riesz_Representation}
Let $\Omega$ be locally compact.
    \begin{enumerate}[\normalfont(a)]
        \item The mapping $\mu \mapsto I_{\mu}$, where $I_{\mu}:C_0(\Omega)\to \R$, is an isometric isomorphism from $\cM(\Omega)$ to $(C_0(\Omega))^*$.
        \item The mapping $\mu \mapsto I_{\mu}$, where $I_{\mu}:C_c(\Omega)\to \R$, is a surjective isometry from $\cM(\Omega)$ to $(C_c(\Omega))^*$.
    \end{enumerate}
\end{theorem}

We also note the following straightforward result that sheds light on the relationship between parts (a) and (b) in Theorem \ref{thm:Riesz_Representation}. It follows directly from the Stone-Weierstraß Theorem
\ref{thm:stoneWeierstrass} and the triangle inequality.

\begin{proposition}
\label{prop:vague convergence Cc vs C0}
Let $\Omega$ be locally compact and $\{\mu_n\}\cup \{\mu\} \subset \cM(\Omega)$ with $\sup_{n \in \N}\norm{\mu_n} < \infty$. Then
\begin{equation}\label{eq:extended_vague}
    I_{\mu_n}(f) \to I_\mu(f)  \textnormal{ for all } f \in C_0(\Omega) \quad \text{ if and only if } \quad I_{\mu_n}(f) \to I_\mu(f) \textnormal{ for all } f \in C_c(\Omega).
\end{equation}

\end{proposition}

Given that one can find a variety of definitions for vague convergence in the extant literature, some remarks on our definition are in order.

\begin{remark}
(a) Our definition of vague convergence is the most common one found in the literature; see e.g.~Berg et al \cite[Chapter 2]{berg_harmonic_1984}, Dieudonné  and Macdonald \cite[Section XIII.4]{dieudonne_treatise_1970}, Kallenberg \cite[Chapter 5]{kallenberg_foundations_2021} or Klenke \cite[Section 13.2]{klenke_probability_2014}.

(b) In a setting where $\Omega$ is locally compact and motivated by \Cref{thm:Riesz_Representation}, vague convergence is defined for test functions in $C_0(\Omega)$ (rather than in $C_c(\Omega)$) by Folland \cite[Section 7.3]{folland_real_1999}. However, in light of Proposition \ref{prop:vague convergence Cc vs C0}, this stronger definition coincides with our definition if the sequence of measures is uniformly bounded.

(c) When $\Omega$ is a Polish space (i.e., complete and separable), the vague topology on $\cM^+(\Omega)$ (which characterises vague convergence) has alternatively been defined to be generated by the family of mappings $\pi_f:\cM^+(\Omega) \to \R_+$ where $f $ are nonnegative continuous functions with metric bounded support. This is the approach taken by Kallenberg \cite[Section 4.1]{kallenberg_random_2017} and Daley and Vere-Jones \cite[Section A2.6]{daley_introduction_2003}. Basrak and Planinić \cite{basrak_note_2019} show that this definition coincides with our definition using the theory of boundedness due to Hue \cite{hu_introduction_1966}. Moreover, \cite{basrak_note_2019} show explicitly that these vague topologies make $\cM^+(\Omega)$ a Polish space in its own right. In particular, this latter fact convinces us that our definition is the most natural one.
\end{remark}

\subsection{Organisation of the paper}
The remainder of the paper is organised as follows.
\cref{section:Vague_convergence} describes the relationship between vague and weak convergence in $\cM(\Omega)$, including the weak and vague convergence of the positive and negative parts in the Hahn-Jordan decomposition. The results are summarised in \Cref{table:1}. In the special case that $\Omega = \R$, \cref{section:BV} studies the relationship between the vague convergence of a sequence of measures $\{\mu_n\} \subset \cM(\R)$ and the pointwise convergence of their distribution functions $\{F_{\mu_n}\}$.
 \Cref{section:appendix} contains some auxiliary results needed in the main body of the paper.


\section{Relationship between vague and weak convergence}\label{section:Vague_convergence}

We first revisit the direct relationship between weak and vague convergence for signed measures. As a warm-up, we recall that vague convergence allows for a loss of mass in the limit, while weak convergence does not.

\begin{example}\label{example:weak_vs_vague_1}
    Let $\mu$ be the zero measure and $\{\mu_n\} \subset \cM(\R)$ be such that $\mu_n := \da_{n} - \da_{-n}$, where for $x \in \R$, $\da_{x}$ denotes the Dirac measure at $x$.  Then $\textnormal{v-lim}_{n \to \infty}\mu_n=\mu$ since for any $f \in C_c(\R)$, \begin{equation*}
        \lim_{n \to \infty} I_{\mu_n}(f) = \lim_{n\to \infty} (f(n) - f(-n)) = 0 = I_{\mu}(f).
    \end{equation*}
Moreover, it holds that $\lim_{n \to \infty}\mu_n(\R) = \mu(\R)$,
i.e. the \textit{signed mass} is preserved.

Now take  $f \in C_b(\R)$ such that
\begin{equation*}
    f(x) = \begin{cases}
        x \hspace{1.5cm}\textnormal{ for } x
        \in (-1,1),\\
        \textnormal{sign}(x)\hspace{0.5cm}\textnormal{ otherwise},
    \end{cases}
\end{equation*}
Thus, we do not have $ \textnormal{w-lim}_{n \to \infty} = \mu$ since
    \begin{equation*}
        2 = \lim_{n \to \infty}I_{\mu_n}(f) \neq \lim_{n \to \infty}I_{\mu}(f) = 0.
    \end{equation*}
\end{example}

Intuitively, what goes wrong in Example \ref{example:weak_vs_vague_1} is that mass is “sent to infinity”. The precise condition that avoids this is \emph{tightness}.

\begin{definition}
A sequence $\{\mu_n\} \subset \cM(\Omega)$ is called \emph{tight} if any $\e>0$ there exists a compact set $K_\e \subset \Omega$ such that
\begin{equation}
\label{def:tight}
\sup_{n \in \N}\modd{\mu_n}(K_\e^c) \leq \e.
\end{equation}
\end{definition}

\begin{remark}
\label{rem:tight}
Since each $\mu \in \cM(\Omega)$ is tight by inner regularity of $\modd{\mu}$, we can replace \eqref{def:tight} by
\begin{equation}
\label{eq:tight:limsup}
\limsup_{n \to \infty} \modd{\mu_n}(K_\e^c) \leq \e.
\end{equation}
\end{remark}
Tightness is exactly the condition that lifts vague to weak
convergence for positive measures. This remains true for signed
measures. The proof of the next result follows from Prohorov's
theorem for signed measures, see \Cref{thm:Prohorov}.\footnote{A
direct proof of \Cref{prop:vague_to_weak_strict}(a) follows also
from a generalisation of \cite[Lemma
5.20]{kallenberg_foundations_2021}.}

\begin{proposition}\label{prop:vague_to_weak_strict}
    Let $\{\mu_n\} \cup \{\mu\}\subset \cM(\Omega)$.
    \begin{enumerate}[\normalfont(a)]
        \item If $\vague_{n \to \infty}\mu_n = \mu$ and $\{\mu_n\}$ is tight, then  $\textnormal{w-}\lim_{n \to \infty}\mu_n = \mu$.
        \item  If $\textnormal{w-}\lim_{n \to \infty}\mu_n = \mu$, then $\vague_{n \to \infty}\mu_n = \mu$. If in addition $\Omega$ is Polish (i.e., complete and separable), then $\{\mu_n\}$ is tight.
    \end{enumerate}
\end{proposition}

    If $\Omega$ is locally compact, the heuristic that vague convergence ignores mass “being sent to infinity” leads us to note that vague convergence in $\cM(\Omega)$ (without loss of signed mass) can be viewed as weak convergence in $\cM(\Omega_\infty)$, where $\Omega_\infty$ denotes the one-point compactification of $\Omega$; see \Cref{def:alexandroff_topology}. To this end, note that a measure $\mu \in \cM(\Omega)$ can be canonically extended to a measure $\mu^\infty \in \cM(\Omega_\infty)$ by setting $ \mu^\infty(A) := \mu(A)$ for $A \in \sB(\Omega)$ and $\modd{\mu^\infty}( \{\infty \}) :=0$. We then have the following result, which follows directly from Proposition \ref{prop:vague convergence Cc vs C0} and \Cref{thm:alexandroff_topology}.

    \begin{proposition}\label{prop:vague_to_weak_compactification}
        Let $\Omega$ be locally compact and $\{\mu_n\}\cup \{\mu\} \subset \cM(\Omega)$ with $\sup_{n \in \N} \Vert \mu_n \Vert < \infty$. Denote by $\mu^\infty_n$ and $\mu^\infty$ the canonical extension of $\mu_n$ and $\mu$, respectively. Then $\textnormal{v-}\lim_{n \to \infty }\mu_n = \mu$ and $\mu_n(\Omega) \to \mu(\Omega)$ if and only if $\textnormal{w-}\lim_{n \to \infty }\mu^\infty_n = \mu^\infty$.
    \end{proposition}

    \begin{remark} For \emph{signed} measures, weak convergence in $\cM(\Omega_\infty)$ is strictly weaker than weak convergence in $\cM(\Omega)$. Indeed, \Cref{example:weak_vs_vague_1} gives an example of  $\{\mu_n\}\cup \{\mu\} \subset \cM(\Omega)$ with $\sup_{n \in \N} \Vert \mu_n \Vert < \infty$ such that $\vague_{n \to \infty}\mu_n = \mu$ and $\mu_n(\Omega) \to \mu(\Omega)$ (and hence $\textnormal{w-}\lim_{n \to \infty }\mu^\infty_n = \mu^\infty$), but $\weak_{n \to \infty}\mu_n \neq \mu$.
    \end{remark}

    We next investigate under which conditions vague convergence implies the convergence of the positive and negative parts in the Hahn--Jordan decomposition. The following result shows that the necessary and sufficient extra condition is that no mass is lost on compact sets.

       \begin{proposition}\label{prop:vague_convergence_of_pm}
            Let $\Omega$ be locally compact and $\{\mu_n\}\cup \{\mu\} \subset \cM(\Omega)$. Then $\vague_{n \to \infty}{\mu_n^\pm} = {\mu}^\pm$ if and only if $\vague_{n \to \infty}\mu_n = \mu$ and
                \begin{equation}\label{eq:vague_convergence_of_pm_compact_condition}
                    \limsup_{n \to \infty} \modd{\mu_n}(K) \leq \modd{\mu}(K).
                \end{equation}
                for every compact set $K \subset \Omega$.
        \end{proposition}

        \begin{proof}
     First, suppose that $\vague_{n \to \infty}{\mu_n^\pm} = {\mu}^\pm$. Then clearly $\vague_{n \to \infty}\mu_n = \mu$, and \eqref{eq:vague_convergence_of_pm_compact_condition} is satisfied due to the Portmanteau Theorem in the form of \Cref{thm:portmanteau_extension}\normalfont(b).

        Conversely, suppose that $\vague_{n \to \infty}\mu_n = \mu$ and \eqref{eq:vague_convergence_of_pm_compact_condition} is satisfied. By \Cref{thm:appaendix_open}, for every open set $\Theta \subset \Omega$,
        \begin{equation*}
            \liminf_{n \to \infty} \modd{\mu_n}(\Theta) \geq \modd{\mu}(\Theta).
        \end{equation*}
        Thus, \Cref{thm:portmanteau_extension}\normalfont(b)  gives $\vague_{n \to \infty}\modd{\mu_n} = \modd{\mu}$. Now $\vague_{n \to \infty}{\mu_n}^\pm = {\mu}^\pm$ follows by noting that
        \begin{equation*}
            \mu_n^+ =\half(\modd{\mu_n}+\mu_n)\quad \textnormal{and}\quad\mu_n^- =\half(\modd{\mu_n}-\mu_n). \qedhere
        \end{equation*}
        \end{proof}

        Note that Condition \eqref{eq:vague_convergence_of_pm_compact_condition} does not restrict  “total mass being lost at infinity”.  By imposing an additional restriction to  mitigate this possibility, we can strengthen Proposition \ref{prop:vague_convergence_of_pm} to deduce that $\weak_{n \to \infty}{{\mu_n}^\pm} = {{\mu_n}^\pm}$.

        \begin{proposition}\label{prop:weak variation from vague}
Let $\Omega$ be locally compact and $\{\mu_n\}  \cup \{\mu\}\subset \cM(\Omega)$. Then $\weak_{n \to \infty}{\mu_n^\pm} = {\mu}^\pm$ if and only if $\vague_{n \to \infty}\mu_n = \mu$ and
                \begin{equation}\label{eq:prop:weak variation from vague}
\limsup_{n\to \infty}\norm{\mu_n} \leq \norm{\mu}.
                \end{equation}
        \end{proposition}

        \begin{proof}
First,  suppose that $\weak_{n \to \infty}{\mu_n^\pm} = {\mu}^\pm$. Then $\weak_{n \to \infty}{\mu_n} = {\mu}$ and $\weak_{n \to \infty}{|\mu_n|} = {|\mu|}$. This implies in particular that  $\vague_{n \to \infty}{\mu_n} = {\mu}$ and
            \begin{equation}\label{eq:lim_agrees}
                \lim_{n\to \infty}\norm{\mu_n} =     \lim_{n\to \infty} \int_\Omega \de |\mu_n| = \int_\Omega \de |\mu| =
\norm{\mu}.
            \end{equation}

        Conversely, suppose that  $\vague_{n \to \infty}\mu_n = \mu$ and \eqref{eq:prop:weak variation from vague} is satisfied. By Propositions \ref{prop:vague_convergence_of_pm} and \ref{prop:vague_to_weak_strict}, it suffices to show that \eqref{eq:vague_convergence_of_pm_compact_condition} is satisfied and the sequence $\{\mu_n\}$ is tight.

        First, we establish \eqref{eq:vague_convergence_of_pm_compact_condition}. Seeking a contradiction, suppose there exists a compact set $K \subset \Omega$ such that
        \begin{equation}
        \label{eq:pf:prop:weak variation from vague:ineq 01}
            \limsup_{n \to \infty} \modd{\mu_n}(K) > \modd{\mu}(K).
        \end{equation}
        Since $K^c$ is open, it follows from Theorem \ref{thm:appaendix_open} that
        \begin{equation}
        \label{eq:pf:prop:weak variation from vague:ineq 02}
            \liminf_{n \to \infty} \modd{\mu_n}(K^c) \geq \modd{\mu}(K^c).
        \end{equation}
        Adding \eqref{eq:pf:prop:weak variation from vague:ineq 01} and \eqref{eq:pf:prop:weak variation from vague:ineq 02}, it follows that
        \begin{align*}
        \limsup_{n \to \infty} \norm{\mu_n} &= \limsup_{n \to \infty} \modd{\mu_n} (\Omega) \geq \limsup_{n \to \infty} \modd{\mu_n}(K) +  \liminf_{n \to \infty} \modd{\mu_n}(K^c)  > \modd{\mu}(\Omega)= \norm{\mu},
        \end{align*}
        and we arrive at a contradiction to \eqref{eq:prop:weak variation from vague}.

Next, we show that the sequence $\{\mu_n\}$ is tight.  Let $\e > 0$. By inner regularity of $\mu$, there exists a compact set $K  \subset \Omega$ such that $\modd{\mu}(K^c) \leq \e$.  By local compactness of
        $\Omega$, there exists an open set $L \supset K$ such that its closure $\overline{L} =: K_\e$ is compact. Using \eqref{eq:prop:weak variation from vague} and Theorem \ref{thm:appaendix_open}, we obtain

        \begin{align*}
            \limsup_{n \to \infty}\modd{\mu_n}(K_\e^c) &=  \limsup_{n \to \infty}\curvebrac{ \norm{\mu_n} - \modd{ \mu_n}(K_\e) }\leq \limsup_{n \to \infty}\curvebrac{ \norm{\mu_n} - \modd{\mu_n}(L) } \\
        &\leq \norm{\mu} - \liminf_{n \to \infty}\modd{ \mu_n}(L)\leq \norm{\mu} - \modd{\mu}(L)\leq  \norm{\mu}  -
            \modd{\mu}(K)  = \mu(K^c)\leq \e. \qedhere
        \end{align*}
\end{proof}

    To summarise, starting from vague convergence $\vague_{n \to \infty}\mu_n = \mu$,
    \Cref{prop:vague_to_weak_strict} tells us that we get $\weak_{n \to \infty}\mu_n = \mu$ if
    mass is not “lost at infinity”.  \Cref{prop:vague_convergence_of_pm} asserts that
    if mass is not “lost on compact sets”, then we get $\vague_{n \to \infty}\mu_n^\pm = \mu^\pm$. Finally, \Cref{prop:weak variation from vague} tells us that if mass is not  “lost globally”, then we even get $\weak_{n \to \infty}\mu_n^\pm = \mu^\pm$. These results are summarised in \Cref{table:1}.

    \begin{table}[h!]
        \centering
        \caption{{$\Omega$ is a  ({Polish$^\star$, locally compact$^{\star\star}$})
            metrisable space and $\{\mu_n\}\cup \{\mu\} \subset \cM(\Omega)$. }
            }
        \label{table:1}
        \begin{tabular}{||c  c  c||}
         \hline
         \textbf{Condition(s) A} &  &   \textbf{Condition(s) B}  \\ [0.5ex]
         \hline\hline
         \makecell{$\vague_{n \to \infty} \mu_n = \mu$,\\ and $\forall \e > 0$, $\exists$ compact set $K_\e$\\ such that $\limsup_{n \to \infty} \modd{\mu_n}(K_\e^c) \leq \e$} & \makecell{$\mathbf{\Rightarrow}$ \\ ${\stackrel{\star}{\Leftarrow}}$}  &$\weak_{n \to \infty} \mu_n = \mu$ \\
         \hline
         \makecell{$\vague_{n \to \infty} \mu_n = \mu$,\\ and $\forall$ compact $K \subset \Omega$ \\ $\limsup_{n \to \infty}\modd{\mu_n}(K) \leq \modd{\mu}(K)$} & ${\stackrel{\star\star}{\Leftrightarrow}}$  & $\vague_{n \to \infty} {\mu_n}^\pm = {\mu}^\pm$  \\ [1ex]
         \hline
         \makecell{$\vague_{n \to \infty} \mu_n = \mu$,\\ and $\limsup_{n \to \infty}\norm{\mu_n}\leq \norm{\mu}$} & ${\stackrel{\star\star}{\Leftrightarrow}}$   & $\weak_{n \to \infty} {\mu_n}^\pm = {\mu}^\pm$   \\ [1ex]
         \hline
        \end{tabular}
    \end{table}

\section{Vague convergence and convergence of distribution functions}\label{section:BV}
{In this section, we study the special case that $\Omega = \R$ (with the usual order topology) and link vague convergence on $\R$ to the pointwise convergence of their distribution functions (at continuity points of the limiting measure). To this end, we first need to introduce some further pieces of notation.

Let $\BV(\R)$ denote the space of all functions of bounded variation on $\R$. For $F \in \BV(\R)$ and $x \in \R$, denote the total variation of $F$ on $(-\infty, x]$ by $\bV_F(x)$ and set $F^\uparrow(x) := \half(\bV_{F}(x) + F(x))$ and $F^\downarrow(x) := \half(\bV_{F}(x) - F(x))$. Note that $\bV_F, F^\uparrow, F^\downarrow :\R \to [0,\infty]$ are nondecreasing functions.

For any $\alpha \in \R$ and $\mu \in \cM(\R)$, the \emph{distribution function of $\mu$, centred at $\alpha$,} is the function $F_{\mu}^{(\alpha)} \in \BV(\R)$ defined by
\begin{equation*}
    F_{\mu}^{(\alpha)}(x) := \begin{cases}
        \hspace{.25cm}\mu((\alpha,x])\quad \textnormal{if }x \geq \alpha,\\
        -\mu((x,\alpha])\quad\textnormal{if }x < \alpha.
    \end{cases}
\end{equation*}
Note that $F_{\mu}^{(\alpha)}$ is right-continuous, and for any $a \leq b$ with $a, b \in \R$,
\begin{equation}
\label{eq:dist:meas:rel}
    F_{\mu}^{(\alpha)}(b) - F_{\mu}^{(\alpha)}(a) = \mu((a,b]).
\end{equation}

The relationship \eqref{eq:dist:meas:rel} between distribution functions and signed measures is bijective, which follows from the following result; for a proof see \cite[Theorem 5.13]{leoni_first_2017}.

\begin{theorem}\label{thm:BVloc_to_g_signed_measure}
    Let $F \in \BV(\R)$ be right-continuous. Then there exists a unique $\mu_F\in \cM(\R)$ such that
    \begin{equation*}
        \mu_F((a,b]) = F(b) - F(a)
    \end{equation*}
    for all $a \leq b$ with $a,b \in \R$.  Moreover, $\modd{\mu_F} = \mu_{\bV_F}$.
\end{theorem}
}

Let $[-\infty,\infty]$ be the (affine) extended real line (with the order topology). Any $\mu \in \cM(\R)$ can canonically be extended to $\cM([-\infty,\infty])$ by setting $\modd{\mu}(\{\pm \infty\}) := 0$.  Similarly, for $\alpha \in \R$,  $F^{(\alpha)}_{\mu}$ can canonically be extended to $[-\infty,\infty]$ by setting $F^{(\alpha)}_\mu(\pm\infty) := \lim_{x \to \pm \infty}F^{(\alpha)}_\mu(x)$. Finally, we can define $F^{(-\infty)}, F^{(+\infty)} \in \BV(\R)$ by
\begin{equation*}
F_\mu^{(-\infty)}(x) := \mu((-\infty,x]) \quad \text{and} \quad F_\mu^{(+\infty)}(x) := -\mu((x,\infty)), \quad x \in \R,
\end{equation*}
respectively, which again can canonically be extended to $[-\infty,\infty]$. Note that $F_\mu^{(-\infty)}$ is usually called the distribution function of $\mu$ and denoted by $F_\mu$.

Last but not least, we say that $x \in \R$ is a \emph{continuity point} of $\mu \in \cM(\R)$ if $\mu(\{x\}) = 0$.

\subsection{Relationship between the convergence of distribution functions and vague convergence}

We start our discussion on the relationship between the convergence of distribution functions and vague convergence by recalling the key result for \emph{positive} measures. This type of result is essentially known -- at least under stronger conditions, see e.g.~\cite[Proposition 7.19]{folland_real_1999}. It will follow as a corollary of our main result, Theorem \ref{thm:vague_F_equivalence} below.

\begin{theorem}
\label{thm:vague:pos}
Let $\{\mu_n\}\cup \{\mu\}\subset \cM^+(\R)$ and $\alpha \in \R$ be a continuity point of $\mu$. Then the following are equivalent:
\begin{enumerate}
\item $F^{(\alpha)}_{\mu_n} \to F^{(\alpha)}_\mu$ at the continuity points of $\mu$.
\item $\vague_{n \to \infty}\mu_n = \mu$.
\end{enumerate}
Moreover, if $\alpha = -\infty$ or $\alpha =+\infty$, the equivalence remains true if we require in addition that $\lim_{K \downarrow -\infty}\squarebrac{\limsup_{n \to \infty} \mu_n((-\infty, K])} =0$  when $\alpha = -\infty$, or $\lim_{K \uparrow \infty}\squarebrac{\limsup_{n \to \infty} \mu_n( (K, \infty))} = 0$ when $\alpha = +\infty$.
\end{theorem}

\begin{remark}
(a) The assumption that $\alpha$ is a continuity point of $\mu$ in Theorem \ref{thm:vague:pos} is necessary.
Indeed, let $\mu_n := \da_{1/n}$ and $\mu := \da_0$. Then $\vague_{n \to \infty}\mu_n = \mu$ but
\begin{equation*}
F^{(0)}_{\mu_n}(x)  = 0 \not \to  -1 = F^{(0)}_{\mu}(x), \quad x < 0.
\end{equation*}

(b) As a sanity check, one notes that if $\{\mu_n\} \cup\{\mu\} \subset \cM^+(\R)$ are probability measures, whence $\limsup_{n \to \infty} \Vert \mu_n \Vert = \Vert \mu \Vert = 1$, then \Cref{thm:vague:pos} together with Proposition \ref{prop:weak variation from vague} shows that $\weak_{n \to \infty} \mu_n = \mu$ if and only if $F^{(-\infty)}_{\mu_n} \to F^{(-\infty)}_\mu$ at all continuity points of $\mu$. This is often shown as a consequence of Portmanteau's theorem for weak convergence.
\end{remark}

Both implications  ``(a) $\Rightarrow$ (b)'' and ``(b) $\Rightarrow$ (a)'' in Theorem \ref{thm:vague:pos} are false for signed measures. The first counterexample shows that $F^{(\alpha)}_{\mu_n} \to F^{(\alpha)}_\mu$ at the continuity points of $\mu$ does \emph{not} imply that $\vague_{n \to \infty}\mu_n = \mu$. It relies on $\{F^{(\alpha)}_{\mu_n}\}$ being unbounded on a compact set.
\begin{example}\label{example:sup_neccesary}
    Let $F_n:\R\to \R$ be supported on $[0,2/n]$ and linear between the points $\{0,1/n,2/n\}$ such that $F(0) := 0 =: F(2/n)$ and $F(1/n) := 2^n$; see \Cref{figure:sup_neccesary} for a clear visualisation. For $n \in \N$, let $\mu_n := \mu_{F_n}$ according to \Cref{thm:BVloc_to_g_signed_measure} and denote by $\mu$ the zero measure. Then for any $x \in \R$, we have $F^{(0)}_{\mu_n}(x) = F^{(0)}_n(x) \to F^{(0)}_\mu(x)$.

    Now take $f \in C_c(\R)$ such that
    \begin{equation*}
        f(x) := \begin{cases}
(1+x)\quad \textnormal{for }x \in [-1,0), \\
            (1-x)\quad \textnormal{for } x \in [0,1], \\
            ~0 \hspace{1.25cm}  \textnormal{for } x \in [-1,1]^c.
        \end{cases}
    \end{equation*}
    Then for $n \geq 2$.
    \begin{align*}
        I_{\mu_n}(f) & = 2^n\curlybrac{\int_{0}^{1/n} (1-x) \de x - \int_{1/n}^{2/n} (1-x)  \de x}= \frac{2^{n+1}}{n^2}.
    \end{align*}
    Thus, $I_{\mu_n}(f) \not \to I_\mu(f) =0$.

\end{example}

\begin{figure}[ht]
    \centering{
    \includegraphics[width=12cm]{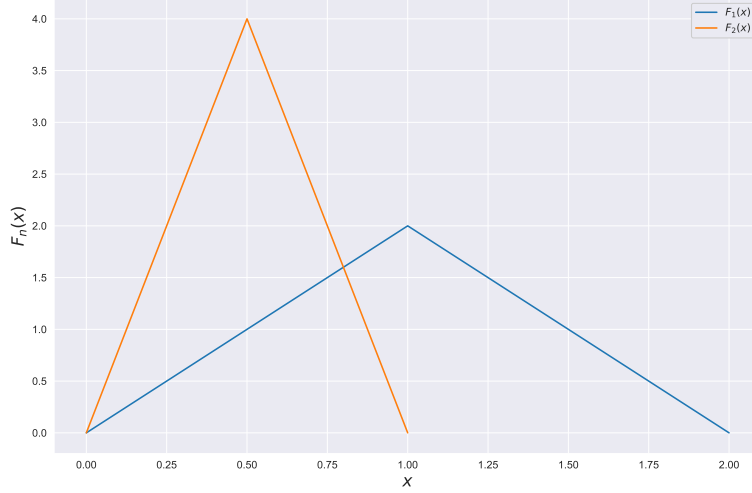}
    \caption{A visualisation of $F_1$ and $F_2$ defined in \cref{example:sup_neccesary}\label{figure:sup_neccesary}.}
    }
\end{figure}

The next counterexample shows that $\vague_{n \to \infty} \mu_n = \mu$  does not imply $F^{(\alpha)}_{\mu_n}\to F^{(\alpha)}_\mu$ at the continuity points of $\mu$ since mass can be lost locally. This happens when the positive and negative parts of the singular decompositions of $\{\mu_n\}$ cancel in the limit.

\begin{example}\label{example:weak_vs_vague_2}
    Let  $\mu_n := \da_0 - \da_{1/n}$, and let $\mu$ be the zero measure. Then it is straightforward to check that $\vague_{n \to \infty} \mu_n = \mu$ (even $\weak_{n \to \infty} \mu_n = \mu$).
   However, we do not have $F^{(0)}_{\mu_n} \to F^{(0)}_\mu$ at all continuity points of $\mu$. Indeed, fix $x > 0$. Then for $n \geq \frac{1}{x}$,
   \begin{equation*}
    F^{(0)}_{\mu_n}(x) = \da_0((0,x])- \da_{1/n}((0,x]) = -1
   \end{equation*}

   \begin{equation*}
    -1 = \lim_{n \to \infty}F^{(0)}_{\mu_n}(x) \neq F^{(0)}_{\mu}(x) = 0.
   \end{equation*}
 \end{example}

 Thus, in order to ensure that the distribution functions converge at continuity points, one must ensure that mass is preserved locally. This motivates the following definition.

\begin{definition}\label{def:no_mass}
    Let $\Omega$ be a metrisable space and $\{\mu_n\} \subset \cM(\Omega)$. We say that  the sequence $\{\mu_n\}$ \emph{has no mass at a point} $x \in \Omega$, if for any $\e > 0$, there exists an open neighbourhood $N_{x,\e}$ of $x$, such that
    \begin{equation*}
        \limsup_{n \to \infty}\modd{\mu_n}(N_{x,\e}) \leq \e.
    \end{equation*}
    In the case where $\Omega = \R$, we say that the sequence $\{\mu_n\}$ has no mass at $ +\infty \textnormal{ (resp.}-\infty\textnormal{)}$, when the family of canonical extensions of $\{\mu_n\}$ has no mass at $ +\infty \textnormal{ (resp.}-\infty\textnormal{)}$.
\end{definition}

\begin{remark}
\Cref{def:no_mass} implies that the family $\{\mu_n\} \subset \cM(\R)$ is tight if and only if it has no mass at $+\infty$ and $-\infty$.
\end{remark}

 For $\{\mu_n\} \subset \cM(\R)$, the preceding discussion leads us to a clear characterisation of vague and weak convergence of $\{\mu_n\}$ from the convergence of $F_{\mu_n}$, and vice versa.

\begin{theorem}\label{thm:vague_F_equivalence}
Let $\alpha \in \R$ and $\{\mu_n\}\cup \{\mu\}\subset \cM(\R)$.
    \begin{enumerate}[\normalfont(a)]
        \item If $F^{(\alpha)}_{\mu_n}(x) \to F^{(\alpha)}_{\mu}(x)$ at all continuity points of $\mu$ and $\{\mu_n\}$ is bounded on compact sets, then $\vague_{n \to \infty}\mu_n = \mu$.
        \item  If  $\vague_{n \to \infty}\mu_n = \mu$, $\alpha$ is a continuity point of $\mu$, and $\{\mu_n\}$ has no mass at the continuity points of $\mu$, then $F^{(\alpha)}_{\mu_n} \to F^{(\alpha)}_\mu$ at the continuity points  of $\mu$.
    \end{enumerate}
Moreover, if $\alpha = -\infty$ or $\alpha =+\infty$, both parts remain true if we require in addition in {\normalfont (a)} that $\{\mu_n\}$ is bounded on compact neighbourhoods of $\alpha$ (in the extended order topology) and in {\normalfont (b)} that $\{\mu_n\}$  has no mass at $\alpha$ (for the canonical extensions of $\{\mu_n\}$).
\end{theorem}

\begin{proof}
We only establish the result for $\alpha \in \R$. The extension of the proof to $\alpha \in \{-\infty, \infty\}$ is straightforward.

    \normalfont(a) First, let $f \in \cC := C^1(\R) \cap C_c(\R)$. Then $f$ is supported by a compact interval $K \subset\R$, and we may assume without loss of generality that $\alpha \in K$. Then \{$F_{\mu_n}^{(\alpha)}\}$ is bounded on $K$ since
    \begin{equation*}
        \modd{F_{\mu_n}^{(\alpha)}(x)} \leq \sup_{n \in \N}\modd{\mu_n}(K) < \infty, \quad  x \in K.
    \end{equation*}
Moreover, $F^{(\alpha)}_{\mu_n}\to F^{(\alpha)}_{\mu}$ a.e.~by the fact that $\mu$ has only countably many atoms. Therefore, an integration by parts and the dominated convergence theorem give
    \begin{equation}
\label{eq:thm:vague_F_equivalence:convergence dist}
        I_{\mu_n}(f) = -\int_K f'(x) F_{\mu_n}^{(\alpha)}(x)~ \den x\rightarrow -\int_K f'(x) F_{\mu}^{(\alpha)}(x)~\den x  = I_{\mu}(f).
    \end{equation}
Next, let  $f \in C_c(\R)\subset C_0(\R)$ and $\e>0$.  Since $\cC$ is a subalgebra of $C_0(\R)$ that separates points and vanishes nowhere, it is dense in $C_0(\R)$ by the Stone-Weierstraß Theorem; see~\cref{thm:stoneWeierstrass}. Thus, there exists $g \in \cC$ such that $\norm{f-g}_\infty < \e$. Then $f$ and $g$ are both supported by some compact interval $L$. Hence, the triangle inequality and \eqref{eq:thm:vague_F_equivalence:convergence dist} give
    \begin{align*}
        \limsup_{n \to \infty}\modd{I_{\mu_n}(f) - I_{\mu}(f)} &\leq \limsup_{n \to \infty}\curvebrac{\modd{I_{\mu_n}(f-g)} + \modd{I_{\mu_n}(f) - I_{\mu}(f)} + \modd{I_{\mu}(f-g)}}\\
        &\leq \curvebrac{\sup_{n \in \N}\modd{\mu_n}(L) + \norm{\mu}}\e.
    \end{align*}
Using that $\{\mu_n\}$ is bounded on compact sets and taking $\e \downarrow 0$ establishes the claim.

(b) Let $t \in \R$ be a continuity point of $\mu$. The case when $t = \alpha$ is trivial, so we may assume without loss of generality that $t > \alpha$, since $F_{\mu}^{(\alpha)}(t) = -F_{\mu}^{(t)}(\alpha)$.

For $\delta > 0$, define the cut-off function $\rho_\delta \in C_c(\R)$ by
\begin{equation*}
\rho_\delta(x) =
\begin{cases}
0 & \text{if } x \notin (\alpha - \delta, t + \delta) \\
\frac{1}{\delta} (x + \delta - \alpha), & \text{if } x \in (\alpha -\delta, \alpha), \\
1 & \text{if } x \in [\alpha, t], \\
\frac{1}{\delta} (t + \delta - x), & \text{if } x \in (t, t+\delta),\\
\end{cases}
\end{equation*}
and for $x \in \R$, the open ball around $x$ of radius $\delta$ by $B_\delta(x)$. Then
    \begin{align}
        &\limsup_{n \to \infty}\curvebrac{\modd{F_{\mu_n}^{(\alpha)}(t) - F_{\mu}^{(\alpha)}(t)}}\nonumber\\
        \leq~&\limsup_{n\to \infty}\Big(\modd{\int(\I_{(\alpha,t]} - \rho)(x)\mu_n(\den x)  }+ \modd{\int \rho(x)\mu_n(\den x)  - \int \rho(x)\mu(\den x) }\nonumber\\
        &\quad +\modd{\int(\I_{(\alpha,t]} - \rho)(x)\mu(\den x)  }\Big) \nonumber \\
        \leq ~& \limsup_{n \to \infty}\Big( \modd{\mu_n}((\alpha - \da,\alpha])  + \modd{\mu_n}((t,t+\da)) + \modd{\mu}((\alpha - \da,\alpha])+ \modd{\mu}((t,t+\da)) \Big) \nonumber\\
        \leq ~& \limsup_{n \to \infty}\modd{\mu_n}(B_\delta(\alpha)) + \limsup_{n \to \infty}\modd{\mu_n}(B_\delta(t))+ \modd{\mu}(\alpha - \da,\alpha]) + \modd{\mu}((t,t+\da)). \label{eq:pf:thm:vague_F_equivalence}
    \end{align}
Now the result follows by taking $\delta \to 0$, noting that the first two terms on the right had side of \eqref{eq:pf:thm:vague_F_equivalence} vanish by the fact that $\{\mu_n\}$ has no mass at $t$ and $\alpha$, whereas the last two terms on the right had side of \eqref{eq:pf:thm:vague_F_equivalence} vanish by $\sigma$-continuity of $\mu$ and the fact that $\alpha$ is a continuity point of~$\mu$.
\end{proof}

We proceed to prove Theorem \ref{thm:vague:pos}, which is in fact a corollary to Theorem \ref{thm:vague_F_equivalence}.

\begin{proof}[Proof of Theorem \ref{thm:vague:pos}]
We only establish the result for $\alpha \in \R$. The extension of the proof to $\alpha \in \{-\infty, \infty\}$ is straightforward.

        ``(a) $\Rightarrow$ (b)''. By \Cref{thm:vague_F_equivalence}(a), it suffices to show that $\{ \mu_n \}$ are bounded on compact sets. So let $K \subset \R$ be a compact set. Then there exists continuity points $b_1, b_2 \in \R$ of $\mu$ such that $K \subset (b_1, b_2]$. By hypothesis, $\lim_{n \to \infty} F^{(\alpha)}_{\mu_{n}}(b) = F^{(\alpha)}_{\mu}(b)$ for $b \in \{b_1, b_2\}$. Moreover, $\mu_{n}((b_1, b_2]) = F^{(\alpha)}_{\mu_{n}}(b_2)-F^{(\alpha)}_{\mu_{n}}(b_1)$ for each $n \in \N$. Thus, by positivity of $\{ \mu_n \}$,
    \begin{equation*}
\limsup_{n \to \infty}\mu_{n}(K) \leq  \lim_{n \to \infty}\mu_{n}((b_1, b_2])  = \lim_{n \to \infty} F^{(\alpha)}_{\mu_{n}}(b_2) -\lim_{n \to \infty} F^{(\alpha)}_{\mu_{n}}(b_1) = F_\mu^{(\alpha)}(b_2) - F_\mu^{(\alpha)}(b_1) < \infty.
    \end{equation*}
    ``(b) $\Rightarrow$ (a)''. By \Cref{thm:vague_F_equivalence}(a), it suffices to sow that $\{ \mu_n \}$ has no mass at the continuity points of $\mu$. So let $x \in \R$ be a continuity point of $\mu$ and fix $\e > 0$. For $\delta > 0$, denote by $B_\delta(x)$ the open ball around $x$ of radius $\da$ and by $\overline{B_{\delta}(x)}$ its closure. By $\s$-continuity of $\mu$, for any $\e >0$ there exists $\da >0$ such that $\mu(\overline{B_\da(x)})\leq\e$. Thus, by \Cref{thm:portmanteau_extension}(b),
    \begin{equation*}
        \limsup_{n \to \infty}\mu_n\curvebrac{B_{\da}(x)} \leq \limsup_{n \to \infty}\mu_n\curvebrac{\overline{B_{\da}(x)}} \leq \mu \curvebrac{\overline{B_{\da}(x)} } \leq \e. \qedhere
    \end{equation*}
\end{proof}

\begin{remark}
The direction  ``(a) $\Rightarrow$ (b)'' in Theorem \ref{thm:vague:pos} (for $\alpha \in \R$) follows also directly from `(a) $\Rightarrow$ (c)'' in the vague Portmanteau Theorem, see \ref{thm:portmanteau_extension}.
\end{remark}

Compared to Theorem \ref{thm:vague:pos},  parts (a) and (b) in Theorem \ref{thm:vague_F_equivalence}  have an extra condition each. One might wonder if either part implies the hypothesis of the other one. We first show by a counterexample that part (a) in Theorem \ref{thm:vague_F_equivalence} does not imply the the hypothesis of part (b).

\begin{example}\label{example:diadic}
For $n \in \N$, let $F_n:\R \to \R$ be supported on $[-2^{-n},2^{-n}]$ and linear between the points $\{k 2^{-2n} :  k\in \{ -2^n, \ldots, 2^{n} \}  \}$ such that
    \begin{equation*}
        F_n\curvebrac{k 2^{-2n}} :=(k~\textnormal{mod}(2)) 2^{-n}, \quad k\in \{ -2^n, \ldots, 2^{n} \};
    \end{equation*}
see \Cref{figure:F_n} for a clear visualisation. Set $\mu_n := \mu_{F_n}$ and let $\mu$ be the zero measure. Then $\{\mu_n\} \subset \cM(\R)$ satisfies the properties:
    \begin{enumerate}[\normalfont(i)]
        \item $\mu_n$ is supported on $[-2^{-n},2^{-n}]$,
        \item $|\mu_n|([-2^{-n},2^{-n}]) = 2$,
        \item $\modd{F^{(0)}_{\mu_n}(x)} \leq 2^{-n}$ for all $x \in \R$.
    \end{enumerate}
It follows that $F^{(0)}_{\mu_n}\to F^{(0)}_{\mu}$ for all $x \in \R$  and $\{\mu_n\}$ is bounded on compact sets but $\{\mu_n\}$ has mass at $0$, which is a continuity point of $\mu$.
\end{example}

\begin{figure}[ht]
    \centering{
    \includegraphics[width=12cm]{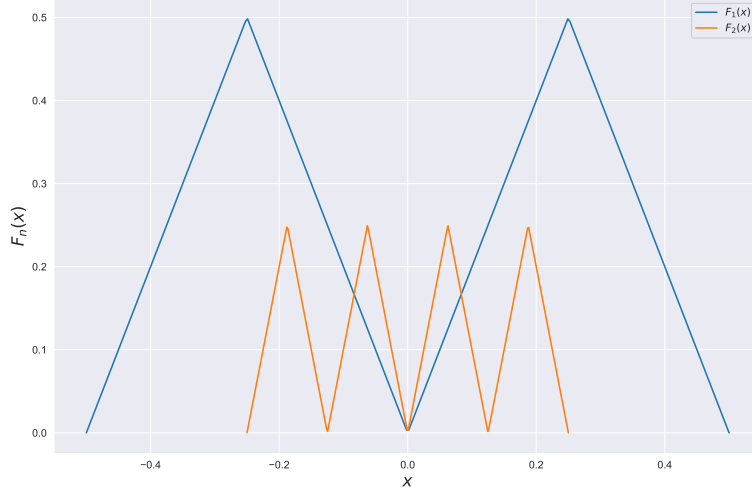}
    \caption{A visualisation of $F_1$ and $F_2$ defined in \cref{example:diadic}\label{figure:F_n}.}
    }
\end{figure}

Fortunately, the assumption that $\{\mu_{n}\}$ has no mass at \emph{any} point is sufficient to establish a proper equivalence result. Note that this slightly stronger assumption is equivalent to the original assumption in the important case that $\mu$ does not have any atoms.
\begin{theorem}\label{corollary:distil}
    Let $\{\mu_{n}\} \cup \{\mu\}\subset \cM(\R)$ and $\alpha \in \R$. Suppose $\{\mu_n\}$ does not have mass on any point of $\R$. Then the following are equivalent:
\begin{enumerate}
\item $F^{(\alpha)}_{\mu_n} \to F^{(\alpha)}_\mu$ at the continuity points of $\mu$.
\item $\vague_{n \to \infty}\mu_n = \mu$.
\end{enumerate}
Moreover,  if $\alpha = -\infty$ or $\alpha =+\infty$, the result remains to true under the additional assumption that $\{\mu_n\}$ has no mass at $\alpha$ (for the canonical extensions of $\mu_n$).
\end{theorem}

\begin{proof}
We only establish the result for $\alpha \in \R$. The extension of the proof to $\alpha \in \{-\infty, \infty\}$ is straightforward.

By Theorem \eqref{thm:vague_F_equivalence}, it suffices to show that the assumption that $\{\mu_n\}$ has no mass on any point of $\R$ implies that $\{\mu_n\}$ is bounded on compact sets. So let $K \subset \R$ be a compact set.

By hypothesis, for each $x \in \R$, the exists an open neighbourhood $N_x$ of $x$ such that $\limsup_{n \to \infty} |\mu_n|(N_x) \leq 1$. Moreover, by compactness, there exists $x_1, \ldots, x_J\in \R$ such that $K \subset \bigcup_{j = 1}^J N_{x_j}$. It follows that
\begin{equation*}
\limsup_{n \to \infty} |\mu_n|(K) \leq \limsup_{n \to \infty} |\mu_n|\bigg(\bigcup_{j = 1}^J N_{x_j}\bigg) \leq \sum_{j =1}^J \limsup_{n \to \infty} |\mu_n|(N_{x_j}) \leq J < \infty. \qedhere
\end{equation*}
\end{proof}

We end this section by noting that the assumption that $\{\mu_n\}$ has no mass at any point of $\R$ is not enough to conclude from $\vague_{n \to \infty}\mu_n = \mu$ that $\vague_{n\to \infty}\modd{\mu_n} = \modd{\mu}$.

\begin{example}\label{example:diadic_2}
For $n \in \N$,  let $F_n:\R \to \R$ be supported on $[-1,1]$ and linear between the points $\{k 2^{-n}: k \in \{ 0,\dots, 2^{n} \}  \}$
    such that
     \begin{equation*}
 F_n\curvebrac{k 2^{-n}} :=(k~\textnormal{mod}(2)) 2^{-n}, \quad k\in \{ -2^n, \ldots, 2^{n} \};
     \end{equation*}
  see \Cref{figure:F_n_2} for a clear visualisation. Set $\mu_n := \mu_{F_n}$ and let $\mu$ be the zero measure. Note that $|\mu_n| = |\mu_1|$ for each $n \in \N$.  Hence it follows trivially that $\vague_{n \to \infty}|\mu_n| =|\mu_1|$. However, using that $\Vert \mu_n \Vert = 2$ and $|F^{(0)}_n| \leq 2^{-n}$ for each $n \in \N$, it follows from \Cref{thm:vague_F_equivalence} that $\vague_{n \to \infty}\mu_n = \mu$. It remains to show that $\{\mu_n\}$ has no mass at any point of $\R$. So fix $x \in \R$ and let $\e > 0$ be given. Let $N_{x,\e}$ be the open ball around $x$ of radius $\e/2$. Then
     \begin{equation}
     \limsup_{n \in \N} \modd{\mu_n}(N_{x,\e}) = \e
     \end{equation}
 \end{example}

\begin{figure}[ht]
    \centering{
    \includegraphics[width=12cm]{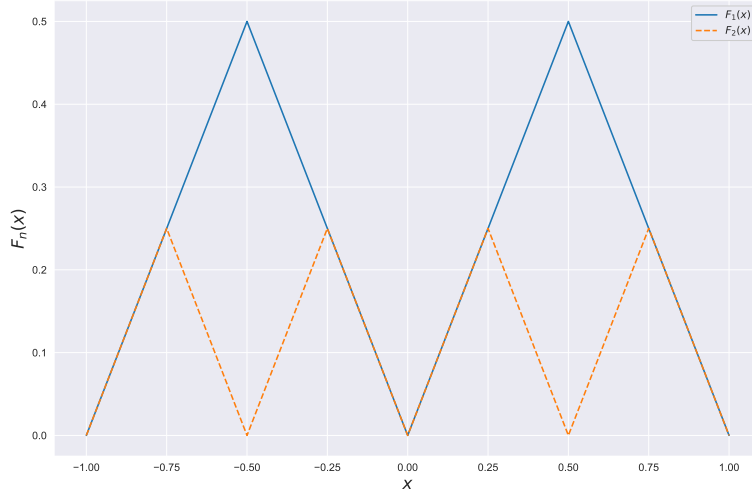}
    \caption{A visualisation of $F_1$ and $F_2$ defined in \cref{example:diadic_2}\label{figure:F_n_2}.}
    }
\end{figure}

\appendix

\section{Appendix}\label{section:appendix}

\subsection{Key results from Functional Analysis and Measure Theory}\label{section:appendix_weak_to_vague}

In this appendix, we collect some key results from Functional Analysis and Measure Theory that are used throughout this paper.

First, we recall the classical Stone-Weierstraß Theorem, see e.g. \cite{de_branges_stone-weierstrass_1959}. To this end, recall that a subset $\cC \subset C_0(\Omega)$ \emph{vanishes nowhere} if for all $x\in \Omega$, there exists some $f \in \cC$ such that $f(x)\neq 0$, and it \emph{separates points} if for each $x,y \in \Omega$ with $x \neq y$, there exists $f \in \cC$ such that $f(x) \neq f(y)$.

\begin{theorem}[Stone-Weierstraß Theorem] \label{thm:stoneWeierstrass}
    Let $\Omega$ be a locally compact Hausdorff space and $\cC$ be a subalgebra of $C_0(\Omega)$. Then $\cC$ is dense in $C_0(\Omega)$ (for the topology of uniform convergence) if and only if it separates points and vanishes nowhere.
\end{theorem}

Next, we state a \emph{vague} version of Portmanteau's Theorem for \emph{positive} measures. While it is very difficult to pinpoint an exact reference, the proof is extremely similar to the weak version; see e.g.\cite[Theorem 13.16]{klenke_probability_2014} and left to the reader.

    \begin{theorem}[Vague Portmanteau Theorem for positive measures]\label{thm:portmanteau_extension}
        Let $\Omega$ be {a locally compact metrisable space} and $\{\mu_n\}\cup \{\mu\} \in \cM^+(\Omega)$. Then the following are equivalent:
        \begin{enumerate}[\normalfont(a)]
            \item $\vague_{n \to \infty}\mu_n = \mu$.
            \item For any compact set $K \subset \Omega$,
            \begin{equation*}
                \limsup_{n \to \infty} \mu_n(K) \leq \mu(K)
            \end{equation*}
            and for any open set $\Theta \subset \Omega$,
            \begin{equation*}
                \liminf_{n \to \infty} \mu_n(\Theta) \geq \mu(\Theta).
            \end{equation*}
                \item For any set $A \subset \Omega$ such that $A \subset K$ for some compact set $K$ and ${\mu}(\partial A) =0 $,
            \begin{equation*}
                \lim_{n \to \infty}{\mu_n}(A) = {\mu}(A).
            \end{equation*}
        \end{enumerate}
        \end{theorem}

One part of the direction ``(a) $\Rightarrow$ (b)'' in the vague Portmanteau theorem (Theorem \ref{thm:portmanteau_extension} extends to signed measure. This result is attributed to Varadarajan; see \cite{varadarajan_measures_1965}. For the  convenience of the reader, we provide a short modern proof.

\begin{theorem}\label{thm:appaendix_open}
Let $\Omega$ be a locally compact normal Hausdorff space.  Let $\{\mu_n\} \cup \{\mu\}\subset \cM(\Omega)$ and assume that  $\vague_{n \to \infty}\mu_n = \mu$. Then for any open set $\Theta \subset \Omega$,
    \begin{equation}\label{eq:appendix_open}
        \modd{\mu}(\Theta) \leq \liminf_{n \to \infty}\modd{\mu_n}(\Theta).
    \end{equation}
    In particular, $\norm{\mu} \leq \liminf_{n \to \infty}\norm{\mu_n}$.
\end{theorem}
\begin{proof}
Let $\Theta \subset \Omega$ be open and  $\e > 0$. Since $\mu$ is inner regular and $\Omega$ is normal and locally compact, as a consequence of Urysohn's lemma \cite[Lemma 2.46]{aliprantis_infinite_1999}, there exists $f \in C_c(\Omega)$ such that $\modd{f}\leq 1$, $\textnormal{supp}(f) \subset \Theta$ and
\begin{equation*}
    \int f \de\mu \geq \modd{\mu}(\Theta) - \e.
\end{equation*}
Then by vague convergence of $\{\mu_n\}$,
\begin{align*}
    \modd{\mu}(\Theta) - \e &\leq \int f \de \mu = \lim_{n \to \infty}\int f \de \mu_n \leq \liminf_{n \to \infty}\int |f| \de \modd{\mu_n} \leq \liminf_{n \to \infty}\modd{\mu_n}(\Theta)
\end{align*}
Now the result follows by letting $\e \downarrow 0$.
\end{proof}

Finally, we state a version of Prohorov's theorem for \emph{signed} measures.

\begin{theorem}[Prohorov's Theorem]\label{thm:Prohorov}
    Let $\Omega$ be a metrisable space and $\mathbf{M} \subset \cM(\Omega)$ nonempty.
    \begin{enumerate}[\normalfont(a)]
        \item If $\mathbf{M}$ is uniformly bounded and tight, then $\mathbf{M}$ is weakly relatively sequentially compact.
        \item If the space $\Omega$ is Polish and $\mathbf{M}$ is weakly relatively sequentially compact, then  $\mathbf{M}$ is uniformly bounded and tight.
    \end{enumerate}
\end{theorem}

\begin{proof}
    \normalfont(a) Take any $\{\mu_n\} \subset \mathbf{M}$. Since $\mathbf{M}$ is a uniformly bounded and tight sequence, both $\{\mu_n^+\}$ and $\{\mu_n^-\}$ are uniformly bounded and tight. By \cite[Theorem 13.29]{klenke_probability_2014}, it follows that there exists a subsequence $\{n_k\}$ such that $\weak_{k \to \infty}\mu_{n_k}^+ = \nu$, for some positive measure $\nu \in \cM(\Omega)$. Similarly, there exists a subsequence $\{n_{k_l}\} \subset \{n_k\}$ such that $\weak_{l \to \infty}\mu_{n_{k_l}}^- = \eta$, for some positive measure $\cM(\Omega)$. Thus it follows that $\weak_{l \to \infty}\mu_{n_{k_l}} = (\nu - \eta) \in \cM(\Omega)$.

    \normalfont(b) See \cite[Theorem 8.6.2]{bogachev_measure_2007}.
\end{proof}

\subsection{One-point compactification}
In this appendix, we recall the one-point compactification of a non-compact locally compact Hausdorff space.

\begin{definition}
\label{def:alexandroff_topology}
 Let $\Omega$ be a non-compact locally compact Hausdorff space with topology $\boldsymbol{\tau}$. Set $\Omega_\infty := \Omega \cup \{\infty\}$, where $\infty \not \in \Omega$, and let
    \begin{equation*}
        \boldsymbol{\tau}_\infty := \boldsymbol{\tau} \cup \{\Omega_\infty \backslash K : K \subset \Omega \textnormal{ is compact}\}
    \end{equation*}
Then $\Omega_\infty$ (with the topology $\boldsymbol{\tau}_\infty)$ is called the \emph{one-point} compactification of $\Omega$.
\end{definition}

The one-point compactification of a non-compact locally compact Hausdorff space has nice properties; see \cite[Proposition 4.36]{folland_real_1999} for a proof.
\begin{theorem}\label{thm:alexandroff_topology}
    Let $\Omega$ be a non-compact locally compact Hausdorff space. Then $\Omega_\infty$ is a compact Hausdorff space and $\Omega$ is an open dense subset of $\Omega_\infty$. Moreover, $f \in C(\Omega)$ extends continuously to $f_\infty \in C(\Omega_\infty)$ if and only if $f = f_0 +c$ where $f_0\in C_0(\Omega)$ and $c$ is a constant. In this case, the extension satisfies $f_\infty(\infty) =c$.
\end{theorem}

\bibliographystyle{amsplain} 
\bibliography{vague}

\end{document}